\documentclass[11pt]{amsproc}

\usepackage{amsfonts,amssymb}

\usepackage{mathrsfs}
\usepackage[colorlinks=true]{hyperref}

\newtheorem{theorem}{Theorem}
\newtheorem{lemma}{Lemma}

\theoremstyle{definition}

\theoremstyle{remark}

\DeclareMathOperator{\trace}{\mathrm{Tr}}
\newcommand{\R}{\mathbb{R}}
\def\by{y}
\def\br{r}
\def\bs{s}

\begin{document}

\title[cmc order 3]{Algebraic cmc hypersurface of order 3 in Euclidean Spaces}
\author{Oscar Perdomo}

\curraddr{Department of Mathematics\\
Central Connecticut State University\\
New Britain, CT 06050\\}

\email{ perdomoosm@ccsu.edu}

\author{Vladimir G. Tkachev}
\address{Department of Mathematics, Link\"oping University, Sweden}
\email{vladimir.tkatjev@liu.se}
\thanks{}

\date{\today}

\begin{abstract}
We prove that there are not algebraic  hypersurfaces of degree 3 in $\mathbb{R}^n$ with non zero constant mean curvature.
\end{abstract}

\maketitle

\section{Introduction}

Understanding algebraic constant mean curvature -cmc- surfaces in the Euclidean  spaces is a basic question that has very little progress. For the three dimensional space $\mathbb{R}^3$, Perdomo showed in \cite{P1} that there are not degree three algebraic cmc surfaces. Also, for surfaces in the Euclidean space $\mathbb{R}^3$, Do Carmo and Barbosa showed \cite{BD} that if $M=f^{-1}(0)$ with $f$ a polynomial and, $\nabla f$ never vanishes on $M$, then $M$ cannot have cmc unless $M$ is a plane, a cylinder or a sphere. Very little is known on the classification of all  immersed algebraic surfaces in $\mathbb{R}^3$ or the classification of all algebraic cmc hypersurfaces in $\mathbb{R}^n$.

In \cite{P1} it was proven that if $\phi:M\to \mathbb{R}^n$ is an immersion with constant mean curvature  $H\ne0$ and $\phi(M)=f^{-1}(0)$ where $f$ is an irreducible polynomial, and, for at least one point $x_0\in \phi(M)$, $\nabla f(x_0)$ does not vanishes, then
\begin{equation}\label{cmceq}
4(n-1)^2 H^2|\nabla f|^6-(2 |\nabla f|^2 \Delta f-\nabla f^t \,\nabla  |\nabla f|^2)^2 =pf
\end{equation}
with $p$ a polynomial. In this paper, we will say that an irreducible polynomial $f$ defines an \textit{algebraic cmc hypersurface} if $f$  satisfies condition (\ref{cmceq}) with some $H\ne0$. Notice that  $f^{-1}(0)$ may have singular points. If  $x_0$ is a regular point of $f$ then \eqref{cmceq} implies that $f^{-1}(0)$ is a cmc hypersurface in a neighbourhood of $x_0$.  With this definition in mind we show that there are not algebraic cmc hypersurfaces in the Euclidean $n$-dimensional space of degree $3$.

\begin{theorem}\label{th1}
If $f:\mathbb{R}^n\to \mathbb{R}$ is an irreducible polynomial of degree three, then, the zero level set of $f$ cannot be an algebraic cmc hypersurface. That is, $f$ cannot satisfy equation (\ref{cmceq}) with $H\ne0$.
\end{theorem}

Some further remarks are worth making at this point. Recall that
\begin{equation}\label{laplacep}
\Delta_p f:=|\nabla f|^2\Delta f +\frac{p-2}{2} \nabla f^t \,\nabla  |\nabla f|^2=0
\end{equation}
is called the  $p$-Laplace equation. Then for $H=0$, the cmc equation \eqref{cmceq} is closely related to the $1$-Laplace equation,
which appears very naturally in the context of minimal cones, see section~6 in \cite{NTVbook}. The assumption $H\ne0$ is crucial, because otherwise, when $n>3$, there exists irreducible algebraic minimal ($H=0$) hypersurfaces of any arbitrarily higher degree, see for instance \cite{T1}.

On the other hand, it is interesting to compare Theorem~\ref{th1} with a similar situation for polynomial solutions of the general $p$-Laplacian equation, $p\ne2$. It follows from \cite{Lewis1980} and the recent results in \cite{T2}, \cite{Lewis2016} that there are no homogeneous polynomial solutions to \eqref{laplacep} in $\R{}^n$ of degree $d=2,3,4$ for any $n\ge2$, of degree $5$ in $\R^{3}$ and also of any degree in $\R^{2}$.

This work was done while the first author visited Link\"oping University.
He would like to thank the Mathematical Institution of Link\"oping University for hospitality.

\section{Proof of the main result}

Before proving the Theorem~\ref{th1}, let us prove some lemmas.

\begin{lemma}\label{lem1}
If $g:\mathbb{R}^n\to \mathbb{R}$ is a polynomial of degree $3$ and $|\nabla g|$ vanishes anytime $g$ vanishes, then $g=l^3$ where $l$ is a linear function.
\end{lemma}

\begin{proof}
Lemma 2.5 in \cite{T1} states that if $g$ is a cubic irreducible polynomial then $g^{-1}(0)$ contains at least one regular point. Therefore, we conclude that our $g$ cannot be irreducible and then, for some polynomial $u$ of degree 2, we have that $g=uv$ where $v$ is linear. Since $\nabla v$ never vanishes and $\nabla g=u\nabla v+v\nabla u$ we conclude that $u(x)=0$ whenever $v(x)=0$. From the real nullstellensatz theorem  we conclude that $u=wv$ for some polynomial $w$. Thus $g=wv^2$. Applying the same argument that we did above, using the polynomial $w$ instead of the polynomial $v$, we conclude that $v^2$ must be a multiple of $w$. Therefore the lemma follows.
\end{proof}

\begin{proof}[Proof of Theorem~\ref{th1}]
Without loss of generality we can assume that the origin is in an element in the hypersurface, in other words, we assume that  $f(0,\dots,0)=0$. Let us write
$$
f=f_3+f_2+f_1\quad\hbox{and} \quad p=p_9+\dots +p_1+p_0
$$
where $f_i=f_i(x)$ and $p_i=p_i(x)$ are homogeneous polynomials of degree $i$.
Comparing the degree 12 homogeneous part in both sides of equation (\ref{cmceq}), we conclude that
\begin{equation}\label{deg12}
\widetilde{H}^2|\nabla f_3|^6=p_9 f_3, \qquad \text{where }\widetilde{H}=2(n-1)H\ne0.
\end{equation}
Using Lemma \ref{lem1}, we conclude that $f_3=l^3$ where $l$ is linear. Since $f(0)=0$, then $l(0)=0$. By doing a rotation and a dilation of the coordinates, if necessary, we can assume that $l=x_1$. For notation sake, let us denote the coordinates in $\mathbb{R}^n$ as $x=x_1$ and $y=(y_1,\dots,y_{n-1})^t=(x_2,\dots,x_n)^t$. With this notation we have the for the homogenous parts in $f=f_3+f_2+f_1$:
\begin{equation}\label{fff}
\begin{split}
f_3&=x^3\\
f_2&=\by^tA\by+k_0 x^2+\by^t\br\, x\\
f_1&=k_1 x+y^t\bs,
\end{split}
\end{equation}
where $A$ is a symmetric $(n-1)\times (n-1)$ matrix, $\br =(r_1,\dots,r_{n-1})^t$ and $\bs =(s_1,\dots,s_{n-1})^t$.

 At this moment we would like to point out that to finish the theorem it is enough to show that the matrix $A$ is zero. The reason for this is that if $A$ is the zero matrix and we relabel the axis so that the vectors $\br $ and $\bs $ are in the plane spanned by the vectors  $(1,0,\dots,0)^t$ and $(0,1,0,\dots,0)^t$, them, the function $f$ would be a function that depends only on the three variables $x$, $y_1$ and $y_2$ and, by \cite{P1} we already know that there are not algebraic cmc surfaces in $\mathbb{R}^3$.

 Let us continue the proof by showing that the matrix $A$ must be the zero matrix. A direct computation yields the following decomposition into homogeneous parts:
\begin{equation}\label{grad1}
|\nabla f|^2=h_0+h_1+h_2+h_3+h_4,
\end{equation}
where $h_k$ is a homogeneous polynomial of degree $k$ given explicitly by
\begin{eqnarray*}
h_0 &=& k_1^2+|\bs|^2 \\
h_1 &=& 4 \bs^t A\by +4 k_0 k_1 x+2 k_1 (\br^t\by )+2 x \,\br^t\bs  \\
h_2 &=&4 x\, \br^tA\by +4 |A\by|^2 +4 k_0 x\, \br^t\by +(\br^t\by )^2+x^2 (4 k_0^2+6 k_1+|\br|^2)\\
h_3&=&12 k_0 x^3+6 x^2\, \br^t\by \\
h4&=&9 x^4
\end{eqnarray*}
Using this representation, we obtain
\begin{align*}
\Delta_1 f &:=2 |\nabla f| \Delta f-\nabla f^t \nabla|\nabla f|^2\\
&=4 |\nabla f|^2 \trace A +16 (k_0+3 x)(\bs^tA\by + |A\by |^2)-8 \br^tA\by  \left(k_1+\br^t\by -3 x^2\right)\\
& -8 x\, \bs^tA\br -4 x^2\, \br^tA\br -16 x\, \br^t A^2\by -16 \bs^t A^2\by -4 \bs^tA.\bs -16 \by^tA^3\by \\
& -4 x (k_0 x+k_1+\br^t\by )|\br|^2+4 (k_0+3 x)|\bs|^2-4 \br^t\bs  \left(k_1+\br^t\by -3 x^2\right),
\end{align*}
which implies that
$$
\Delta_1(f)\equiv  4 \trace A  |\nabla f|^2 \mod\mathrm{Pol}_3,
$$
where we follow an agreement to write
$$
A\equiv B\mod \mathrm{Pol}_k
$$
if $A-B$ is a polynomial of degree $k$ or less.

It also follows from \eqref{grad1} that
$$
|\nabla f|^4\equiv 81\,x^8 \mod \mathrm{Pol}_7
$$
Thus, using the made observations,  equation (\ref{cmceq}) becomes
\begin{equation}\label{hd}
(p_0+p_1+\dots +p_9) (f_1+f_2+f_3)\equiv  \widetilde{H}^2|\nabla f|^6
-6^4(\trace A)^2x^8\mod \mathrm{Pol}_7.
\end{equation}

Our next goal is by using the decomposition for $|\nabla f|^2$ in terms of the $h_i$'s and and the
expression for $(\Delta_1f)^2-4(n-1)^2H^2|\nabla f|^6$ in terms of the $|\nabla f|^2$ up to order 8, deduce that the matrix $A$ should be the zero matrix. To this end, we consider  \eqref{hd} as a polynomial identity with respect to a variable $x$ over the ring $\R{}[y]$.

A key observation is that since $h_4=Q_0x^4$ and $h_3=Q_2x^2$, where $Q_i$ is a homogeneous polynomial of degree  $i$, one immediately obtains from \eqref{grad1}  the following homogenous decomposition:
\begin{equation}\label{grad3}
\widetilde{H}^2|\nabla f|^6
-6^4(\trace A)^2x^8=L_0x^{12}+L_1 x^{10}+L_2 x^8+L_3 x^6+L_4x^4 \mod \mathrm{Pol}_7.
\end{equation}
Here $L_i$ are homogeneous polynomials of degree  $i$. In particular,
\begin{equation}\label{L0}
L_0=3^6\widetilde{H}^2\ne0.
\end{equation}
Next, identifying  the homogeneous parts of degrees $k$, $8\le k\le 12$, in both sides of \eqref{hd} one obtains respectively
\begin{align}
p_9f_3&=L_0x^{12}\nonumber\\
p_8f_3+p_9f_2&=L_1x^{10}\nonumber\\
p_7f_3+p_8f_2+p_9f_1&=L_2x^{8}\label{eqq}\\
p_6f_3+p_7f_2+p_8f_1\qquad\quad &=L_3x^{6}\nonumber\\
p_5f_3+p_6f_2+p_7f_1\qquad\quad\qquad\quad&=L_4x^{4}.\nonumber
\end{align}
The first two equations yield
$$
p_9=L_0x^9
$$
and
$$
p_8=S_2x^{6},
$$
where $S_2=L_1x^2-L_0f_2$. Arguing similarly, one easily finds that
\begin{align*}
p_7&=S_4x^3, \\
p_6&=(L_3-S_2f_1)x^3-S_4f_2,
\end{align*}
where $S_4=L_2x^2-L_0f_1Sx^3-S_2f_2$. Thus, evaluating the last identity of \eqref{eqq} for  $x=0$ and taking into account that $f_3(0)=p_7(0)=0$ we obtain
$$
p_6(0)f_2(0)=0.
$$
Since
$$
p_6(0)=-S_4(0)f_2(0)=S_2(0)f_2^2(0)=-L_0f_2^3(0),
$$
it follows by \eqref{L0} that $f_2(0)=0$, therefore using \eqref{fff} we obtain $f_2(0)=\by^tA\by=0$ for any $y\in \R^{n-1}$. Since $A$ is symmetric, we have $A=0$. As explained before, after sowing that $A$ vanishes we have that $f$ essentially depends on three variable and therefore it cannot define an algebraic hypersurface with constant mean curvature. The theorem is proved.

\end{proof}



\begin{thebibliography}{11}

\bibitem{BD} Barbosa, J. Lucas M.; do Carmo, Manfredo P.  \emph{On regular algebraic surfaces of $\mathbb{R}^3$ with constant mean curvature}. J. Diff. Geometry {\bf 102} (2016), no 2, 173-178.

\bibitem{Lewis1980}
Lewis J.L., \emph{Smoothness of certain degenerate elliptic equations}, Proc.
  Amer. Math. Soc. \textbf{80} (1980), no.~2, 259--265.

\bibitem{Lewis2016}
Lewis J.L. and Vogel A.
\newblock On $p$-laplace polynomial solutions.
\newblock {\em The Journal of Analysis}, 24(1):143--166, 2016.


\bibitem{NTVbook}
N.~Nadirashvili, V.G. Tkachev, and S.~Vl{\u{a}}du{\c{t}}.
\newblock {\em Nonlinear elliptic equations and nonassociative algebras},
   volume 200 of {\em Mathematical Surveys and Monographs}.
\newblock American Mathematical Society, Providence, RI, 2014.


\bibitem{P1} Perdomo, O  \emph{Algebraic constant mean curvature surfaces in Euclidean space}. Houston J. Math. {\bf 39} (2013), 127-136.

\bibitem{T1} Tkachev, V.G.  \emph{Minimal cubic cones via Clifford Algebras}. Complex Anal. Oper. Theory {\bf 4} (2010), 687-700.

    \bibitem{T2}
Tkachev V.G.
\newblock On the non-vanishing property for real analytic solutions of the
  {$p$}-{L}aplace equation.
\newblock {\em Proc. Amer. Math. Soc.}, 144(6):2375--2382, 2016.


 \end{thebibliography}
 \end{document}